\documentclass[11pt]{article}

\usepackage[utf8]{inputenc}

\usepackage[titletoc,title]{appendix}
\usepackage{authblk}
\usepackage{caption}
\usepackage{url}
\usepackage{amsthm}
\usepackage{float}
\usepackage{hhline}
\usepackage{multicol}
\usepackage{subfiles}
\usepackage{enumerate}
\usepackage[shortlabels]{enumitem}
\usepackage{empheq}
\usepackage{tikz-cd}

\usepackage{thmtools}
\usepackage{thm-restate}

\usepackage{ulem}
\usepackage{mathrsfs}

\usepackage{amsmath,amsfonts,amssymb,mathtools}

\usepackage{graphicx,float}

\usepackage{geometry}
\usepackage{comment}

\usepackage{hyperref}
\usepackage[capitalize]{cleveref}

\author{Charles Arnal\footnote{Université Paris-Saclay, CNRS, Inria, Laboratoire de Mathématiques d'Orsay}}

\newcommand{\eps}{\varepsilon}
\newcommand{\dotp}[1]{\langle #1 \rangle}

\newcommand{\Conv}{\mathrm{Conv}}
\newcommand{\Aff}{\mathrm{Aff}}
\newcommand{\Span}{\mathrm{Span}}

\newcommand{\R}{{\mathbb R}}

\declaretheorem[name=Theorem,numberwithin=section]{thm}

\newtheorem{definition}[thm]{Definition}

\newtheorem{proposition}[thm]{Proposition}
\newtheorem{lemma}[thm]{Lemma}

\begin{document}
\title{The distance function to a finite set is
\\a topological Morse function}
\date{}
\maketitle

\begin{abstract}
In this short note, we show that the distance function to any finite set $X\subset \R^n$ is a topological Morse function, regardless of whether $X$ is in general position.
We also precisely characterize its topological critical points and their indices, and relate them to the differential critical points of the function.
\end{abstract}

\section{Introduction}\label{sec:introduction}

The distance function $d_X$ to a compact set $X \subset \R^n$ and its sublevel sets are central objects of study in computational geometry \cite{Chazal_CS_Lieutier_OGarticle, chazalTopology2006, attali2009stability} and in topological data analysis   \cite{edelsbrunner2022computational, chazal2016structure, otter2017roadmap, carlsson2021topological}.
In general, the function $d_X$ is not particularly regular, even when the set $X$ itself is (e.g. when $X$ is a smooth submanifold).
Nonetheless, $d_X$ does enjoy some weak Morse-like properties: though it is not differentiable, one can consider its \textit{generalized gradient} $\nabla d_X$, as defined in \cite{lieutier2004any}. Let $\Pi_X(z)$ denote the set $\{x\in X : d(z,x) = d(z,X)\}$ of projections of $z$ on $X$, and let $\sigma_X(z)$ denote the projection of $z$ onto the convex hull $\Conv(\Pi_X(z))$. Then
\[ \nabla d_X(z) := \frac{z-\sigma_X(z)}{d_X(z)} \]
if $z\not \in X$, and $\nabla d_X(z) :=0$ if $z\in X$.\footnote{It can be shown that the generalized gradient $\nabla d_X(z)$ as defined here coincides with the projection of $0$ onto the similarly named generalized gradient $\partial d_X(z) \subset \R^n$ as defined by Clarke in \cite{clarke1990optimization}.}
A \textit{differential critical point} of $d_X$ is a point $z \in \R^n$ such that $\nabla d_X(z) = 0$, and a \textit{differential critical value} of $d_X$ is the image by $d_X$ of a differential critical point.
Then it can be shown that the Isotopy Lemma from Morse theory still holds, i.e. that changes in the homotopy type of the offsets $X^t := \{ z\in \R^n :  d_X(z) \leq t \} $  can only occur at critical values of $d_X$ (see e.g. \cite{GroveCriticalPoints}).

On the other hand, there is no analogue to the Handle Attachment Lemma (see \cite{milnor1963morse}), and no simple way to control the potential changes in topology at critical values.
In fact, differential critical values need not even correspond to changes in topology; among other examples, consider the half-circle $X:= \{x^2 +y^2 = 1, x\leq 0\}$. The point $(0,0)$ is a differential critical point of $d_X$, but all offsets $\{X^t\}_{t> 0}$ are homeomorphic.

The case where $X$ is a finite point cloud  is of particular importance, as it features preeminently in applications where one only has access to discrete samplings of sets of interest, rather than to the sets themselves. When $X$ is in general position, the frameworks of continuous selections of functions \cite{Jongen_Pallaschke} or of Min-type functions  \cite{gershkovich1997morse} can be applied to show that the distance function $d_X:\R^n \rightarrow \R$ is a \textit{topological Morse function}:
\begin{definition}[Topological Morse functions \cite{morse_topologically_1959}]
    Let $U\subset \R^n$ be an open set and let $f:U\to \R$ be a continuous function. 
    \begin{itemize}
        \item A point $z\in U$ is said to be a topological regular point of $f$ if there is a homeomorphism $\phi:V_1\to V_2$ between open neighborhoods $V_1$ of $0$ in $\R^n$ and $V_2$ of $z$ in $U$ with $\phi(0)=z$ and such that for all $p=(p_1,\dots,p_n)\in V_1$, 
        \begin{equation}
            f\circ \phi(p) = f(z) + p_1.
        \end{equation}
        \item A point $z\in U$ is said to be a topological critical point of $f$ if it is not a topological regular point of $f$.
        \item A point $z\in U$ is said to be a non-degenerate topological critical point of $f$ of index $m$ if there exist an integer $0\leq i \leq n$ and a homeomorphism $\phi:V_1\to U_2$ between open neighborhoods $V_1$ of $0$ in $\R^n$ and $V_2$ of $z$ in $U$ with $\phi(0)=z$ such that for all $p=(p_1,\dots,p_n)\in V_1$, 
        \begin{equation}
            f\circ \phi(p) = f(z) - \sum_{i=1}^m p_i^2 + \sum_{i=m+1}^n p_i^2.
        \end{equation}
        \item The function $f$ is said to be a topological Morse function if all its topological critical points are non-degenerate.
    \end{itemize}
\end{definition}
Topological Morse functions do satisfy the Handle Attachment Lemma (see e.g. \cite[Theorem 5]{song2023generalized}), which in turns allows for the computation of the homology of the sublevel sets of the function, which is one of the main goals in topological data analysis.
As a result, settings in which $X$ is assumed to be in general position have been extensively studied in the literature \cite{Bobrowski_Adler_2014, Bobrowski_Weinberger_2015, BauerEdelsbrunner2016, BobrowskiOliveira2019, Kergorlay2019RandomC, ReaniBobrowski2024,  bauerRoll}.

On the other hand, no such result was known for an arbitrary point cloud $X\subset \R^n$, though a classification of critical configurations in $\R^2$ can be found in \cite{Siersma}. Distance functions to arbitrary compact sets are not topologically Morse (consider e.g. two parallel segments), and the methods used in the case where $X$ is a point cloud in general position fail without the genericity assumption, as the differential critical points of $d_X$ need not be non-degenerate in the sense of Min-type functions or of continuous selections of functions. 
This was unsatisfactory, as there are many settings where $X$ cannot be expected to be in general position, nor can it be perturbed to be made so, e.g. when $X$ is a subsampling of a manifold (as in  \cite{niyogi2008finding, aamari2019estimating, aamari2023optimal, arnal2023critical, arnal2024wasserstein}).

In this short note, we prove that the distance function to an arbitrary finite set is, in fact, a topological Morse function, and give a description of its topological critical points. Let $\Span(S)$ denote the linear span of a set $S\subset \R^n$. 
\begin{thm}\label{thm:main_thm}
    Let $X\subset \R^n$ be a finite set. Then  $d_X:\R^n \rightarrow \R$ is a topological Morse function.
    Furthermore,
    \begin{itemize}
        \item The points of $X$ are topological critical points of $d_X$ of index $0$.
        \item Let $z \in \R^n\backslash X$ be such that $z\in \Conv(\Pi_X(z))$.
        If there exists $v\in \Span(\Pi_X(z)-z)\backslash\{0\}$ such that $\dotp{v,x-z}\leq 0$ for all $x\in \Pi_X(z)$, then $z$ is a topological regular point of $d_X$.
        Otherwise, $z$ is a topological critical point of $d_X$ of index $\dim(\Span(\Pi_X(z) - z))$.
        \item All other points of $\R^n$ are topological regular points of $d_X$.
    \end{itemize}
\end{thm}
Note that the points $z \in \R^n\backslash X$ such that $z\in \Conv(\Pi_X(z))$ are precisely the differential critical points of $d_X$ that do not belong to $X$.
Hence the topological critical points of $d_X$ are (in general) a strict subset of its differential critical points, and a simple criterion specifies whether a given differential critical point is also topologically critical.
Techniques similar to those used in the proof of \cref{thm:main_thm} could allow for a generalization of the notion of Morse Min-type functions from \cite{gershkovich1997morse}, though this is beyond the scope of this article and might be investigated in future work.


\section{Proof of \cref{thm:main_thm}}

\begin{figure}[ht]
    \centering
    \includegraphics[width=0.31\textwidth]{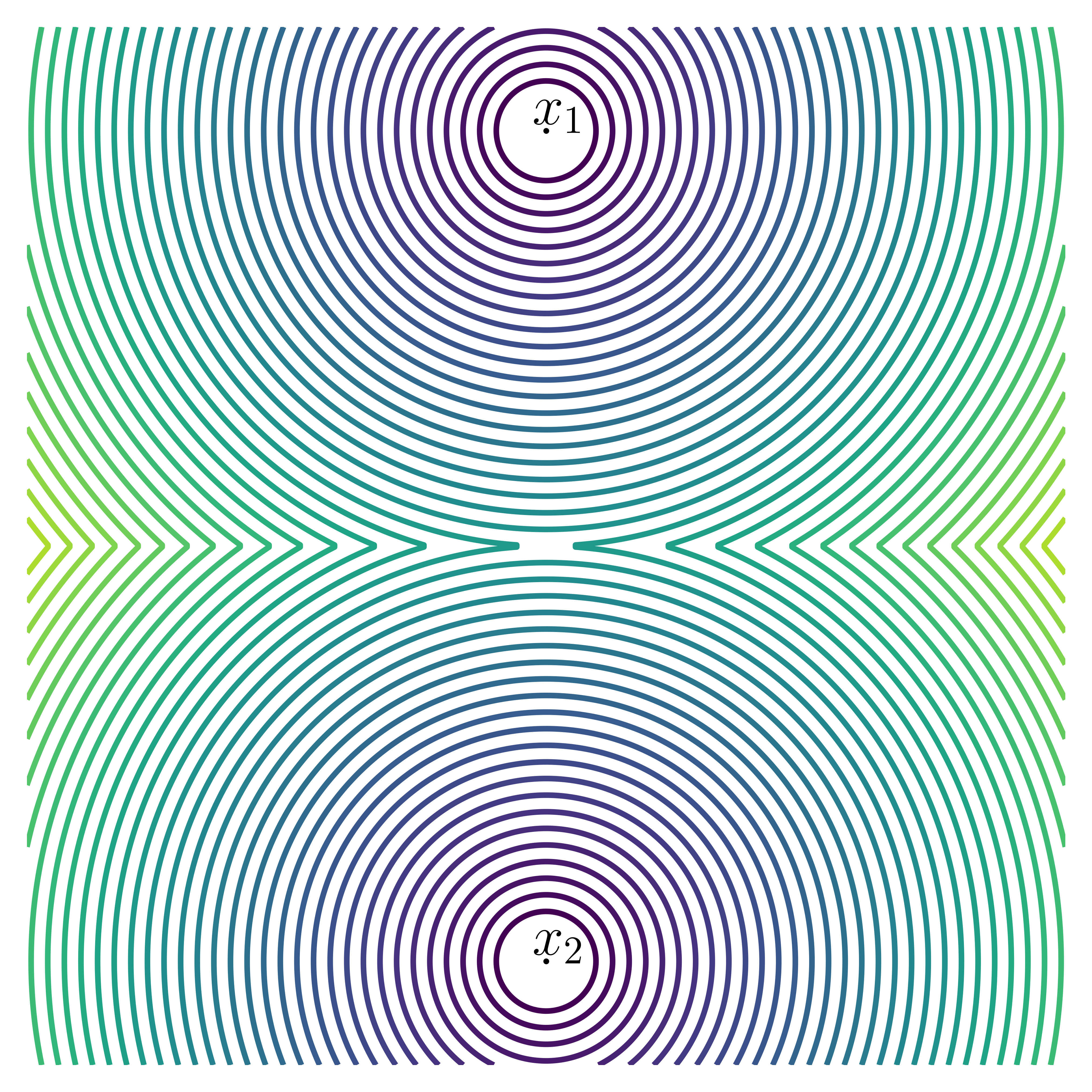}
    \includegraphics[width=0.31\textwidth]{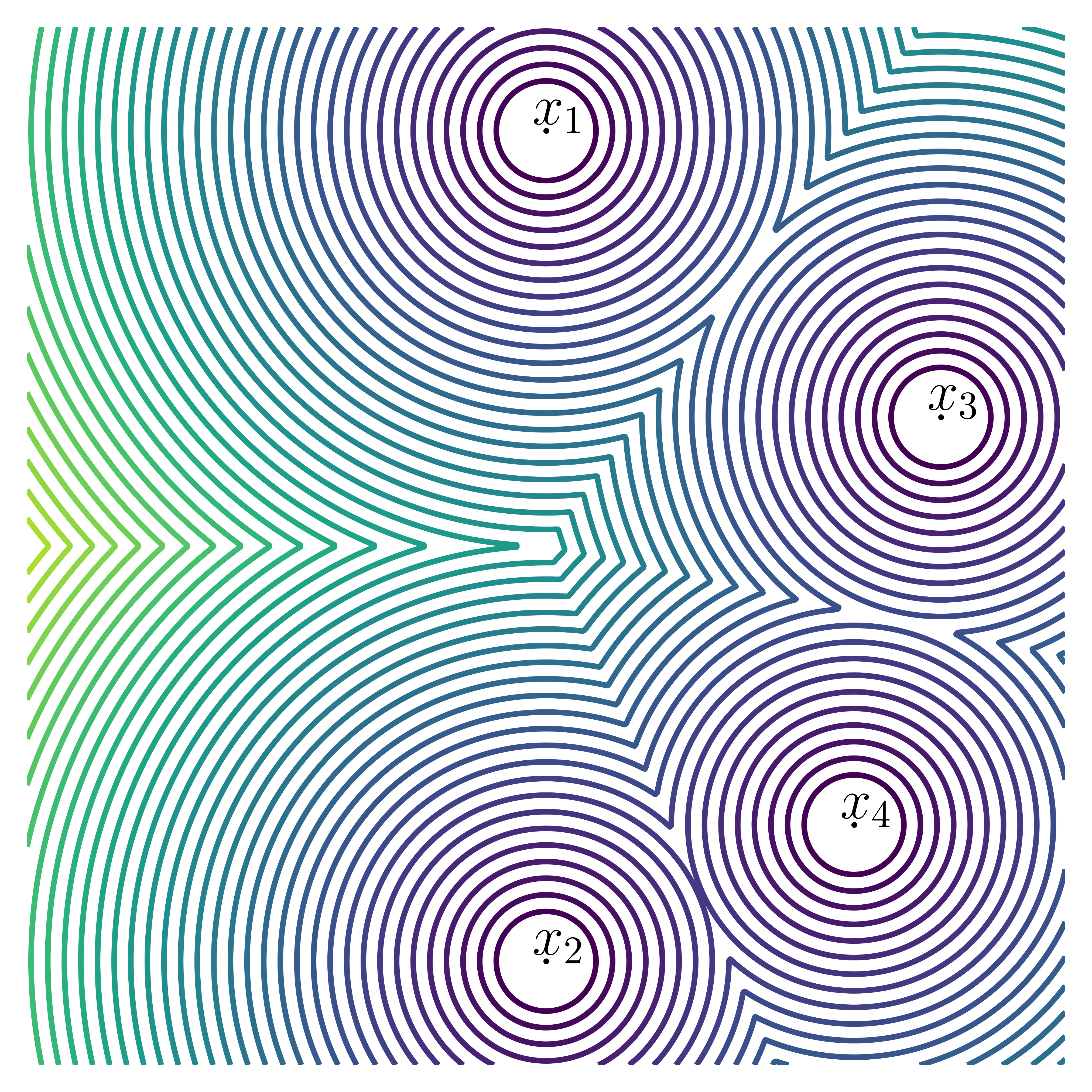}
    \includegraphics[width=0.31\textwidth]{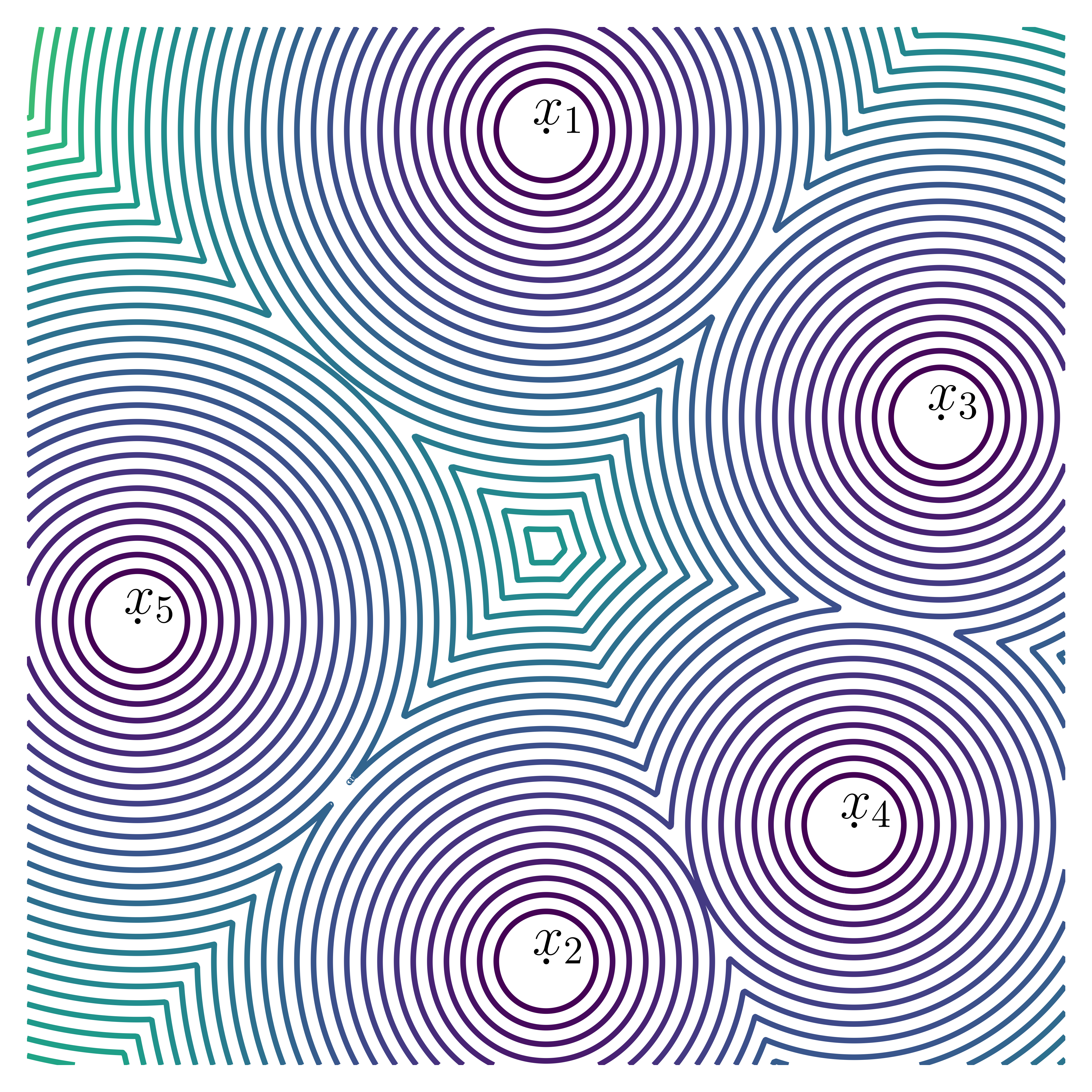}
    \caption{The level sets of the distance function to three sets of points are represented. From left to right, the central point is a non-degenerate topological critical point of index $1$, a topological regular point and a non-degenerate topological critical point of index $2$ respectively.}
    \label{fig:three_configurations}
\end{figure}

Given a point $p\in \R^n$, we let $p_i$ denote its $i$-th coordinate.
We also let $\Aff(S)\subset \R^n$ denote the affine hull of a set $S\subset \R^n$, and $\mathring{S}$ denote its interior.

We split the core of the proof of \cref{thm:main_thm} into two propositions.
The first one corresponds to the cases (in the statement of \cref{thm:main_thm}) where there is no $v\in \Span(\Pi_X(z)-z)\backslash\{0\}$ such that $\dotp{v,x-z}\leq 0$ for all $x\in \Pi_X(z)$; examples of such configurations can be seen on the left and the right of Figure \ref{fig:three_configurations}.

\begin{proposition}\label{prop:local_minimum}
   Let $k\geq 2$ and $X = \{x_1,\ldots,x_k\} \subset \R^m$ be such that $\Aff(X) = \R^m$ and that there exists $z\in \R^m$  with $\Pi_X(z) = X$ and  $z\in \Conv(X)$.
   Assume that there exists no $v\in \R^m\backslash \{0\}$ such that $\dotp{v,x-z}\leq 0$ for all $x\in X$.
   Then $z$ is a topological critical point of $d_X$ of index $m$.
\end{proposition}
\begin{proof}
Let us assume that $z=0$ to simplify notations, and let us write $R = d_X(0)$.
For any $p\in \R^m \backslash \{0\}$, there exists  $x = x(p)\in X$ such that $\dotp{p,x} \geq \dotp{p,x'}$ for all $x'\in X$ (note that $x(p)$ is not necessarily unique - it matters not).
In particular, for any $x'\in X$,
\begin{align*}
\|p-x\|^2 &=\|p\|^2 - 2\dotp{p, x} + \|x\|^2 
\\& \leq     \|p\|^2 - 2\dotp{p, x'} + \|x'\|^2  = \|p-x'\|^2,
\end{align*}
hence $d_X(p) = d(p,x(p))$ (in other words, $p$ belongs to the Voronoi cell of $x(p)$).
Now by hypothesis, 
\[ \lambda:=\min_{\|v\| =1} \max_{x\in X} 2\dotp{v,x} >0. \]
 Let $\eps \in (0,\lambda/2)$, and observe that for any $p\in B(0,\eps) \backslash \{0\}$, we have
 \[d_X(p)^2 =  \|x-x(p)\|^2 =\|p\|^2 - 2\dotp{p, x(p)} + \|x(p)\|^2 \leq \|p\|(\|p\| - \lambda )  +R^2 < R^2.\]
 Hence $0$ is the maximum of $d_X$ over $B(0,\eps)$, and $R-d_X(p) >0$ for any $p\in B(0,\eps)\backslash \{0\}$.
Now consider the continuous map
 \[\Phi: B(0,\eps) \rightarrow \R^m, \ p \mapsto p\frac{\sqrt{R-d_X(p)}}{\|p\|}. \]
Let us show that $\Phi$ is injective.
Indeed, let $p\neq p' \in B(0,\eps)$. If one of those points is $0$ or if $\R_+\cdot p\neq \R_+ \cdot p'$, then clearly $\Phi(p) \neq \Phi(p')$. Otherwise, let $\rho>0$ be such that $p' = \rho p$, and assume without loss of generality that $\rho>1$.
Note that the definition of $x(p)$ above depended only on $p/\|p\|$; hence we can assume that $x(p) = x(p'):=x$.
Then
\begin{align*}
 d_X(p')^2 & = \|p'\|^2 - 2\dotp{p', x} + \|x\|^2  = \rho^2\|p\|^2-2\rho\dotp{p,x} +\|x\|^2 
 \\& = (\rho^2-1)\|p\|^2 - 2(\rho-1)\dotp{p,x} + d_X(p)^2.
\end{align*}
But 
\begin{align*}
(\rho^2-1)\|p\|^2 - 2(\rho-1)\dotp{p,x} & \leq (\rho-1)((\rho+1) \|p\|^2 - \lambda \|p\|) \\
& = (\rho-1)\|p\|( \|p'\| + \|p\| - \lambda ) \leq (\rho-1)\|p\|( 2\eps- \lambda ) <0,
\end{align*}
hence $d_X(p')^2 < d_X(p)^2$ and $\Phi(p)\neq \Phi(p')$.
Thus $\Phi$ is injective, and Brouwer's Theorem on the Invariance of Domain states that $\Phi(B(0,\eps))$ is open, and that $\Phi:B(0,\eps) \rightarrow \Phi(B(0,\eps))$ is a homeomorphism.
Finally, $\|\Phi(p)\|^2 = R-d_X(p)$ for $p\in B(0,\eps)$, hence $ \|q\|^2 = R- d_X\circ \Phi^{-1}(q) $ for $q\in \Phi(B(0,\eps))$ and
\[d_X\circ \Phi^{-1}(q) =R - \sum_{i=1}^m q_i^2, \]
which proves the proposition.

\end{proof}

Our second proposition corresponds to the case where there exists $v\in \Span(\Pi_X(z)-z)\backslash\{0\}$ such that $\dotp{v,x-z}\leq 0$ for all $x\in \Pi_X(z)$; an example of such a configuration can be seen in the middle of Figure \ref{fig:three_configurations}.

\begin{figure}[ht]
    \centering
    \includegraphics[width=0.5\textwidth]{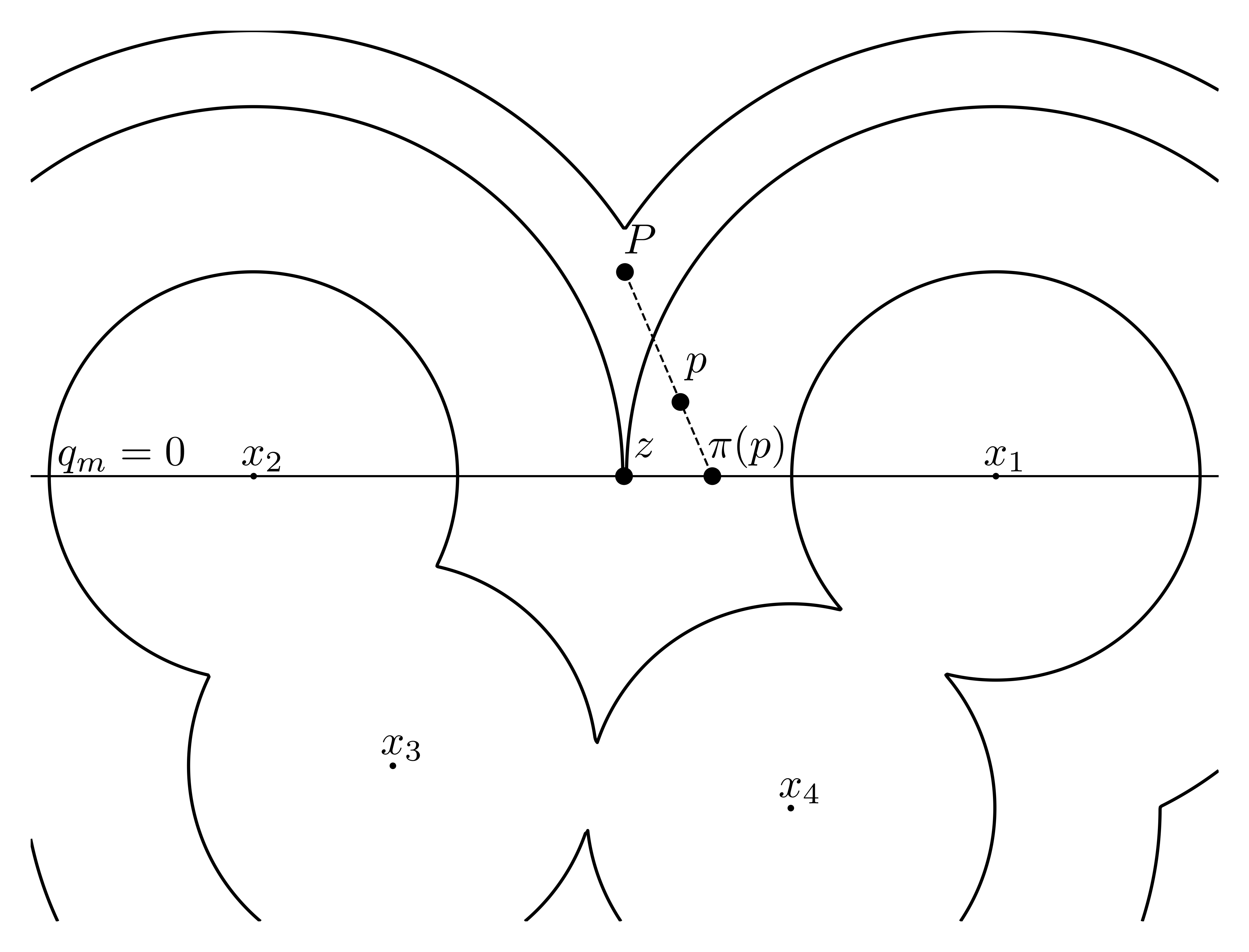}
    \caption{Three level sets of the distance function to the set $X=\{x_1,\ldots,x_4\}$ and the points $P$, $p$ and $\pi(p)$ as defined in the proof of \cref{prop:regular_points}.}
    \label{fig:projection}
\end{figure}

\begin{proposition}\label{prop:regular_points}
   Let $k\geq 2$ and $X = \{x_1,\ldots,x_k\} \subset \R^m$ be such that $\Aff(X) = \R^m$ and that there exists $z\in \R^m$  with $\Pi_X(z) = X$ and  $z\in \Conv(X)$.
   Assume that there exists $v\in \R^m\backslash \{0\}$ such that $\dotp{v,x-z}\leq 0$ for all $x\in X$.
   Then $z$ is a topological regular point of $d_X$.
\end{proposition}
\begin{proof}
Let us assume without loss of generality that $z=0$,  and let us write $R = d_X(0)$.
\cref{lem:existence_of_v} below states that there exists $v\in \R^m\backslash \{0\}$ such that $\dotp{v,x} \leq 0$ for all $x\in X$, and such that $  \max_{x \in X} \dotp{x,w}>0 $
for all $w\in \R^m\backslash\{0\}$ with $\dotp{v,w} \leq 0$.
By compactness, there exists $\lambda >0$ such that
\begin{equation}\label{eq:min_scalar_product}
\min_{\|w\| =1, \dotp{v,w} \leq 0} \max_{x\in X}\dotp{w,x} > \lambda.
\end{equation}
Up to rescaling $v$ and applying an isometric change of coordinates, we can assume that $v$ is equal to the $m$-th standard basis vector $e_m$.
Let $\rho \in (0,\lambda/4) $, and define $P:= \rho e_m$. Let also $\eps \in (0,\min(\rho/2, R/2, \lambda/4))$.

Now for any $p \in B(0,\eps)$, let $\pi(p)$ denote the first $m-1$ coordinates of the projection $p+ (p-P)\frac{p_m}{\rho - p_m}$ of $p$ onto the hyperplane $\{q\in \R^m : q_m = 0\}$ along the direction $p-P$.
Consider the continuous map
\[\Psi:B(0,\eps) \rightarrow \R \times \R^{m-1}, \ p \mapsto (d_X(p)-R,\pi(p)). \]
The situation is illustrated in Figure \ref{fig:projection}.
We are going to show that thanks to our careful choices of $P$ and $\eps$, $\Psi$ is injective.

It is obvious that $\Psi(p)\neq \Psi(p')$ if $p,p' \in B(0,\eps)$ do not belong to the same line passing through $P$.
Thus we only need to show that $d_X$ is injective on the intersection of any line passing through $P$ with $B(0,\eps)$.
Let $p'\neq p''\in B(0,\eps)$ belong to the same line passing through $P$.
As $d_X$ is Lipschitz, Lebourg's Mean Value Theorem \cite[Theorem 2.3.7]{clarke1990optimization} states that there exists $p\in[p',p'']$ such that 
\begin{equation}\label{eq:mean_value}
d_X(p'') - d_X(p') \in \{ \dotp{u, p'' - p'} : \ u\in\partial d_X(p)\}, 
\end{equation}
where $\partial d_X(p)$ is Clarke's generalized gradient for $d_X$ at $p$.
Theorem 2.5.1 from \cite{clarke1990optimization} then shows that
\[\partial d_X(p) =  \Conv\left(\left\{ \frac{p-x}{\|p-x\|} : \ x\in\Pi_X(p) \right\}\right).\]
As $p,p',p''$ belong to the same line passing through $P$, we have $p''-p' = \mu(P-p)$ for some $\mu \in \R^*$. Without loss of generality, we can assume that $\mu>0$.
Let $x\in \Pi_X(p)$;
as in \cref{prop:local_minimum}, $x$ is necessarily such that $\dotp{p,x} = \max_{x'\in X} \dotp{p,x'}$.

If $p_m \leq 0 $, then Inequality \eqref{eq:min_scalar_product} states that $\dotp{p,x} > \lambda \|p\|$, hence
\begin{align*}
  \dotp{P-p,p-x} & = \dotp{P,p} -\dotp{P,x} -\|p\|^2 + \dotp{p,x} > -\|P\|\|p\|  - \|p\|^2 +   \lambda \|p\| 
  \\ & \geq \|p\|( -\rho - \eps +\lambda) \geq \lambda\|p\| /2,
\end{align*}
where we also use the fact that $\dotp{P,x} \leq 0$ by construction.
Likewise, if $p_m>0$, then
\begin{align*}
  \dotp{P-p,p-x}  = & \dotp{P -p,p} -  \dotp{P -p,x}
  \\ \geq   & \rho p_m -\|p\|^2  -  \dotp{P -p_me_m ,x} + \dotp{p-p_me_m,x}.
\end{align*}
But Inequality \eqref{eq:min_scalar_product} applied to $p-p_me_m$ states that $ \dotp{p-p_me_m,x} \geq \lambda\|p-p_me_m\|$.
As $\dotp{e_m, x} \leq 0$, we also have that
\[-\dotp{P -p_me_m ,x}  = - (\rho-p_m) \dotp{e_m ,x} \geq  - (\rho-\eps) \dotp{e_m ,x}\geq 0.\]
Hence
\begin{align*}
  \dotp{P-p,p-x} \geq   & \rho p_m -\|p\|^2  +  \lambda\|p-p_me_m\| \geq \rho \|p\| -\|p\|^2 \geq \rho \|p\|/2,
\end{align*}
where we use the facts that $\rho\leq \lambda$, that $p_m + \|p-p_me_m\| \geq\|p\| $ and that $\|p\|\leq \eps < \rho/2$.
By combining the cases $p_m\leq 0$ and $p_m>0$, we find that there exists $C>0$ such that
\begin{align*}
  \dotp{P-p,p-x} \geq C \|p\|.
\end{align*}
Thus for any $x\in X$ such that $d(x,p) = d_X(p)$,
\[\dotp{p-x, p''-p'} = \dotp{p-x,\mu(P-p)} \geq C\mu \|p\|,\]
and any $u\in \partial  d_X(p) =  \Conv\left(\left\{ \frac{p-x}{\|p-x\|} : \ x\in\Pi_X(p) \right\}\right)$ satisfies
\[ \dotp{u, p''-p'}\geq  \frac{1}{\|p-x\|}C\mu \|p\| \geq \frac{C\mu}{2R}\|p\|.\]
Equation \eqref{eq:mean_value} then states that
\[d_X(p'') - d_X(p') \geq \frac{C\mu}{2R}\|p\|. \]
This is enough to immediately conclude that $d_X(p'') \neq d_X(p')$ when the segment $[p',p'']$ does not contain $0$ (as in that case $\|p\|>0$). When it does, additional elementary arguments yield the same conclusion.
Hence we have proved the injectivity of $d_X$, and thereby that of $\Psi$.

Brouwer's Theorem on the Invariance of Domain then states that $\Psi(B(0,\eps))$ is open, and that $\Psi:B(0,\eps) \rightarrow \Psi(B(0,\eps))$ is a homeomorphism.
As $\Psi(p)_1 = d_X(p)-R$ for any $p\in B(0,\eps)$,
\[ d_X\circ \Psi^{-1}(q) = R+ q_1\]
for any $q \in \Psi(B(0,\eps))$, which proves the proposition.

\end{proof}

Let us now prove the Lemma used in the proof of \cref{prop:regular_points}.
Remember that the dual of a cone $C\subset \R^m$ is defined as $C^* = \{v\in\R^m : \ \dotp{v,u}\geq C \ \forall u \in C\}$.

\begin{lemma}\label{lem:existence_of_v}
    Let $X \subset \R^m$ be such that $\Aff(X) = \R^m$. Assume that there exists $v_0\in \R^m\backslash \{0\}$ such that $\dotp{v_0,x} \geq 0$ for all $x\in X$.
Then there exists $v\in \R^m\backslash\{0\}$ such that $\dotp{v,x}\leq 0$ for all $x \in X$ and that
\[ \dotp{v,w} \leq 0 \ \Rightarrow  \ \max_{x \in X} \dotp{x,w}>0 \]
for all $w\in \R^m\backslash\{0\}$.
\end{lemma}
\begin{proof}
    Consider the closed convex cone $ K := \{w\in \R^m : \ \dotp{w,x} \geq 0 \ \forall x \in X\} $. By hypothesis, we know that $K\neq \{0\}$, as $v_0 \in K$.
    Consider now the dual cone $K^*$ of $K$.
    The cone $K^*$ is \textit{salient}, i.e. it does not contain any non-trivial linear subspace, as otherwise the set $X$ would be included in the orthogonal of such a subspace, contradicting the hypothesis that $\Aff(X) = \R^m$.
    It is known that the interior of the dual of any closed, convex and salient cone is non-empty  (e.g. as a consequence of \cite[Theorems 6.2 and 14.1]{Rockafellar1970}), hence the convex cone $C := K^*$ has non-empty interior.
    Remember that the dual of the dual of a  closed convex cone is equal to the cone itself (see e.g. \cite[Chapter 14]{Rockafellar1970}); thus $C^* = (K^*)^* = K$.
    
    We can now apply Lemma \ref{lem:conic_lemma} below to the convex cone $C$ and its dual $C^* = K$; it states that $\mathring C \cap C^* =\mathring {( K^*)} \cap K \neq \emptyset $.
    Let $v$ be such that $-v \in \mathring {( K^*)} \cap K$. As $K \neq \{0\}$ implies that $0\not \in \mathring {( K^*)}$, $v$ is necessarily non-zero.
    Since $-v\in K$, we have $\dotp{v,x} \leq 0$ for all $x\in X$.
    Furthermore, the fact that $-v$ belongs to the interior of $K^*$ directly implies that $\dotp{v,w} <0 $ for all $w\in K\backslash\{0\}$.
    Let us show that this means that $\dotp{v,w} \leq 0 \Rightarrow  \max_{x \in X} \dotp{x,w}>0 $ for all $w\in \R^m \backslash\{0\}$, which is enough to conclude the proof.

    Indeed, if $w \in \R^m \backslash \{0\}$ is such that $\max_{x \in X} \dotp{x,w}\leq 0$, then $\dotp{x,-w} \geq 0$ for all $x\in X$, hence $-w \in K\backslash\{0\}$. As shown above, this means that $\dotp{v,-w} <0 $, or equivalently that $\dotp{v,w} >0$.
    By contraposition, this proves our claim.

\end{proof}

\begin{lemma}\label{lem:conic_lemma}
Let $C\subset \R^m$ be a convex cone whose interior $\mathring C$ is non-empty, and assume that $C^* \neq \{0\}$. 
Then $\mathring C \cap C^* \neq \emptyset$.


\end{lemma}
\begin{proof}
Suppose that  $\mathring C \cap C^* = \emptyset$. Then $\mathring C$ is the interior of a convex set, and as such it is a non-empty convex open set. The dual cone $C^*$ is a non-empty convex set. Hence we can apply Hahn-Banach's Separation Theorem, which states that there exists $v\in \R^m$ such that 
\begin{equation}\label{eq:first_HB_inequality}
    \sup_{u\in \mathring C} \dotp{v, u} \leq \inf_{w\in C^*} \dotp{v, w}
\end{equation}
and
\begin{equation}\label{eq:second_HB_inequality}
 \dotp{v, u} < \inf_{w\in C^*} \dotp{v, w}   
\end{equation}
for all $u\in \mathring C$; in particular, $v\neq 0$.
As $\inf_{w\in C^*} \dotp{v, w} \in \{0,-\infty\}$, Inequality \eqref{eq:second_HB_inequality} implies that $\inf_{w\in C^*} \dotp{v, w} =0$.
This, in turn, means that $\sup_{u\in \mathring C} \dotp{v, u} = 0$.

Consequently, the vector $-v$ is such that $\inf_{u\in \mathring C} \dotp{-v, u} = 0$.
As $\overline{\mathring C} = \overline{C} \supseteq C$ (this holds for any convex set), and as $ u\mapsto \dotp{-v,u}$ is continuous, 
we have $\inf_{u\in  C} \dotp{-v, u} = 0$.
This means that $-v \in C^*$, but we have shown above that  $\inf_{w\in C^*} \dotp{v, w} =0$, hence $\dotp{v,-v} = -\|v\|^2 \geq  0$, which contradicts the fact that $v\neq 0$.
    
\end{proof}

We can now prove \cref{thm:main_thm}. It is essentially a combination of Propositions \ref{prop:local_minimum} and \ref{prop:regular_points}.

\begin{proof}[Proof of \cref{thm:main_thm}]
By definition, the only differential critical points of $d_X$ are the points $z\not \in X$ such that $z\in \Conv(\Pi_X(z))$, and the points $x\in X$.
The fact that differential regular points of $d_X$ are also topologically regular is shown in \cite[Section 7]{arnal2023critical} (see also \cite{GroveCriticalPoints}), which proves the last point of the statement of \cref{thm:main_thm}.

If $x\in X$, then $d_X$ coincides with $d_{\{x\}}$ on a small ball $B(x,\eps)$ around $x$.
Consider the map $\Xi : B(x,\eps) \rightarrow B(x,\eps^2), p \mapsto x+ \frac{(p-x)}{
\sqrt{\|p-x\|}}$.
It is a homeomorphism, and $d_X\circ \Xi^{-1}(q) = \sum_{i=1}^n q_i^2$ for all $q\in B(x,\eps^2)$, hence $x$ is a non-degenerate topological critical point of $d_X$ of index $0$. This proves the first point of the statement of the Theorem.

Now let $z\in \R^n\backslash X$ be such that $z\in \Conv(\Pi_X(z))$.
On a small open neighborhood $W$ of $z$, the functions $d_X$ and $d_{\Pi_X(z)}$ coincide.
Up to an isometric change of coordinates, we can assume that $z=0$ and that $\Aff(\Pi_X(z)) = \R^m \times \{0\} \subset \R^n$ for some $m\geq 1$ (note that with this change of coordinates, $\Span(\Pi_X(z) - z) = \Aff(\Pi_X(z))$).
Let us also write $R:= d_X(z)$.

If there exists no $v\in \Aff(\Pi_X(z))\backslash\{0\}$ such that $\dotp{v,x}\leq 0$ for all $x\in \Pi_X(z)$, then \cref{prop:local_minimum} applies to the set $\Pi_X(z)$ and the subspace $\Aff(\Pi_X(z))\cong \R^m$, and there exist open neighborhoods $U_1,U_2\subset \R^m$ of $0$ and a homeomorphism $\Phi_1 : (U_1,0)\rightarrow (U_2,0)$ such that $d_{\Pi_X(z)}\circ \Phi_1^{-1} (q) = R -\sum_{i=1}^mq_i^2$ for all $q\in U_2$.
We define the homeomorphism
\begin{align*}
    \Phi_2 : U_2 &\longrightarrow  \Phi_2(U_2) =: U_3,
    \\ p & \longmapsto  \frac{p}{\|p\|}\sqrt{R^2-(R-\|p\|^2)^2},
\end{align*}
and we write $p = (p^m,p^{n-m})\in \R^{m}\times\R^{n-m}$ for all $p \in \R^n$.
Consider the homeomorphism
\begin{align*}
    \Phi_3 : U_1\times B^{n-m}(0,\eps) &\longrightarrow  U_3\times B^{n-m}(0,\eps),
    \\ (p^m,p^{n-m}) & \longmapsto  (\Phi_2\circ \Phi_1(p^m), p^{n-m}),
\end{align*}
where $B^{n-m}(0,\eps)$ is the ball of radius $\eps$ centered at $0$ in $\R^{n-m}$.
For any $p=(p^m,p^{n-m})$ close to $0$, we have $d_X(p) = d_{\Pi_X(z)}(p) = \sqrt{(d_{\Pi_X(z)}(p^m))^2 + \|p^{n-m}\|^2 }$, hence
\begin{align*}
    d_X\circ \Phi_3^{-1}(q) 
    & =\sqrt{(d_{\Pi_X(z)}\circ \Phi_1^{-1} \circ \Phi_2^{-1}(q^m))^2 + \|q^{n-m}\|^2 }
    \\ & = \sqrt{(R - \| \Phi_2^{-1}(q^m)\|^2)^2 + \|q^{n-m}\|^2 }
    \\ & =\sqrt{(R - (R -\sqrt{R^2 - \|q^m\|^2}))^2 + \|q^{n-m}\|^2 }
     \\ & =\sqrt{R^2 - \|q^m\|^2 + \|q^{n-m}\|^2 } = \sqrt{R^2 - \sum_{i=1}^mq_i^2+ \sum_{i=m+1}^nq_i^2 }.
\end{align*}
Up to making the sets $U_1,U_2,U_3$ smaller, Morse's Lemma then states that there exists a diffeomorphism $\Phi_4:U_3\times B^{n-m}(0,\eps) \rightarrow \Phi_4(U_3\times B^{n-m}(0,\eps) ) =:U_4$ that maps $0$ to $0$ such that 
$d_X\circ \Phi_3^{-1}\circ \Phi_4^{-1}(q) = R -\sum_{i=1}^mq_i^2+ \sum_{i=m+1}^nq_i^2.$
Hence $z$ is a non-degenerate topological critical point of index $m = \dim(\Span(\Pi_X(z)-z))$ of $d_X$.

Likewise, if there exists $v\in \Aff(\Pi_X(z))\backslash\{0\}$ such that $\dotp{v,x}\leq 0$ for all $x\in \Pi_X(z)$, then \cref{prop:regular_points} applies to the set $\Pi_X(z)$ and the subspace $\Aff(\Pi_X(z))\cong \R^m$, and there exist open neighborhoods $V_1,V_2\subset \R^m$ of $0$ and a homeomorphism $\Psi_1 : V_1\rightarrow V_2$ such that $d_{\Pi_X(z)}\circ \Psi_1^{-1} (q) = R + q_1$ for all $q\in V_2$.
As in the case above, one can define (using $\Psi_1$) a continuous local change of coordinates $\Psi$ defined on an open neighborhood of $0$ in $\R^n$  such that 
$d_{\Pi_X(z)}\circ \Psi (q) = R + q_1$ (the Implicit Function Theorem is applied instead of Morse's Lemma).

Hence $z$ is a topological regular point: this completes the proof of the theorem.
\end{proof}

\section*{Acknowledgements}
The author would like to thank Pierre-Louis Blayac, Vincent Divol, David Cohen-Steiner and Fred Chazal for helpful discussions.

\bibliography{biblio_new}

\end{document}